\documentclass[12pt]{article}
\usepackage{amsthm}
\usepackage{amsmath}
\usepackage{tikz}
\usetikzlibrary{arrows}
\usepackage{amssymb}
\usepackage{enumitem}
\usepackage[margin=1in]{geometry}
\usepackage{amsthm} \usepackage[colorlinks=true,linkcolor=blue]{hyperref}
\hypersetup{
  citecolor={blue}
} 
\usepackage{amssymb}
\usepackage{enumitem}

\usepackage{amsthm}
\usepackage{amsmath}
\usepackage{tikz}
\usetikzlibrary{arrows}
\usepackage{amssymb}
\usepackage{enumitem}

\usepackage{float}

\newtheorem{theorem}{Theorem}[section]

\newtheorem{corollary}[theorem]{Corollary}
\newtheorem{lemma}[theorem]{Lemma}
\newtheorem{definition}[theorem]{Definition}

\newtheorem{question}[theorem]{Question}

\newcommand{\RR}{\mathbb R}

\include{epsf}

\begin{document}
\title{Exotic almost complex circle actions on $6$-manifolds}

\author{Panagiotis Konstantis and Nicholas Lindsay}

\maketitle
 
\abstract{Jang has proven a remarkable classification of $6$-dimensional manifolds having an almost complex circle action with $4$ fixed points. Jang classifies the weights and associated multigraph into six cases, leaving the existence of connected manifolds fitting into three of the cases unknown. We show that one of the unknown cases may be constructed by a surgery construction of Kustarev, and the underlying manifold is diffeomorphic to $S^4 \times S^2$. We show that the action is not equivariantly diffeomorphic to a linear one, thus giving a new exotic $S^1$-action of on a product of spheres that preserves an almost complex structure. We also prove a uniqueness statement for the almost complex structures produced by Kustarev's construction and prove some topological applications of Jang's classification. }

\section{Introduction}

In \cite[Theorem 1.1]{Ja}, Jang proved a remarkable classification of $6$-dimensional manifolds having an almost complex circle action with $4$ fixed points. This generalized earlier classifications by Ahara when the manifold satisfies $c_{1}^3 \neq 0$, and $\frac{c_{1}c_{2}}{24}=1$ \cite{A} and Tolman in the case when the action is Hamiltonian \cite{To2}.

In \cite[Theorem 1.1]{Ja}, such manifolds are classified into 6 cases, depending on their weights and associated multigraph. The existence of connected, almost complex $6$-manifolds with circle actions fitting into case 4-6 of the classification was left open. These three new unknown cases are characterized by the fact that they have vanishing Todd genus.  Jang points out that a disconnected example fitting into case 4 of the classification is furnished by $S^6 \cup S^6$. This case stands out from the others because the associated multigraph is disconnected, in fact the disjoint union two copies of the GKM graph associated to $S^6$.

Firstly, we observe that a connected example may be obtained via surgery construction of Kustarev \cite{Ku}.

\begin{theorem} \label{main} There is a connected, almost complex $6$-manifold with an almost complex $S^1$ action fitting into case 4 of Jang's classification. It is obtained by performing the Kustarev surgery operation on two copies of $S^6$, and is diffeomorphic to $S^4 \times S^2$. 
\end{theorem}

The proof follows from a topological result which expresses the fiber connect sum with a sphere along a subsphere as a surgery operation, see Section \ref{proofsec} for the relevant definitions. The most technichal step is to determine the class of the framing associated to a free orbit, which is done in subsection \ref{comfram}.

Moreover, we prove the following corollary, showing that these circle actions are exotic.

\begin{corollary} \label{exotic}
There exists a circle action on $S^4 \times S^2$, preserving an almost complex structure and not equivariantly diffeomorphic to a linear circle action.
\end{corollary}

Our second result, involves the more general surgery construction of Kustarev for almost complex $2n$-manifolds with an almost complex $T^{k}$-action \cite{Ku2}. Kustarev proves that for $2n-k=2,4,5,6 \mod 8$, the the original construction may be generalized directly. I.e. the result of gluing two such manifolds along the complement of open neighborhoods of free orbits has a $T^{k}$-invariant almost complex structure.

Our next result is the following:

\begin{theorem}
Suppose that $2n-k = 4,5 \mod 8$, then the almost complex structure produced by Kustarev's surgery is unique up to homotopy. In particular, for $S^1$-actions the homotopy class of the almost complex structure is unique, provided $n= 3 \mod 4$.
\end{theorem}

In particular, for Kustarev's original surgery construction for $S^1$-actions on $6$-manifolds, the homotopy class is uniquely determined by the original almost complex $S^1$-manifolds. In the following section we return our attention to Jang's classification, and prove the following topological restriction on examples fitting into another of the unknown cases.

\begin{theorem}
Let $M$ be a simply connected, $\mathbb{Z}$-equivariantly formal, almost complex $6$-manifold fitting into Case 6 of Jang's classification. Then, $M$ is diffeomorphic to the quadric $3$-fold. 
\end{theorem}

The proof follows from the ABBV localization theorem and the Wall-Jupp-Zubr classification theorem for $6$-manifolds.

We finish the article with a brief section discussing open questions and further directions. Because of the wealth of examples produced via Kustarev's surgery construction it is necessary to restrict the class of manifolds in Jang's classification to hope to obtain finiteness up to for example diffeomorphism (homotopy equivalence could be equally appropriate here). We pose some specific questions, which appear to be interesting and perhaps approachable.

\textbf{Acknowledgements.} We would like to thank Silvia Sabatini for helpful discussions.

\section{Preliminaries}
We start by recalling some standard results about almost complex circle actions, which we will need.

\begin{lemma} \label{Degreeformula}
	Let $M$ be a closed, almost complex $6$-manifold with a almost complex $S^{1}$-action having finitely many fixed points. Label the points $p_1 \ldots p_k$ and let the weights at $p_j$ be denoted $\{w_{j,1}, \ldots , w_{j,n}\}$. Then, $$\int_M c_1(TM)^3  = \sum_{p_j}\frac{( \sum_j w_{i,j})^3}{\prod_j w_{i,j}}.$$
\end{lemma}
\begin{proof}
	This follows by applying Applying ABBV localization \cite{AB,BV} 
theorem to $(c_{1}^{S^{1}}(TM))^3$.
\end{proof}

\begin{lemma} \label{prel}
	Let $M$ be closed $6$-manifold having a almost complex $S^{1}$-action with finitely many fixed points. Let $N_{0}$ denote the number of fixed points such that none of the weights are negative. Then,
		 $$N_{0} = \frac{c_{1}c_{2}(M)}{24} = \mathrm{Td}(M),$$ is the Todd genus of $M$.
\end{lemma}
\begin{proof}
This follows from the localization formula for the Hirzebruch $\chi_{y}$-polynomial 
\cite[Section 5.7]{HBJ},
 of which the Todd genus is the constant term.
\end{proof}

The following theorem is the main result of \cite{Ja}.

\begin{theorem}[{\cite[Theorem 1.1]{Ja}}] \label{jang} 
	Let $M$ be an almost complex $6$-manifold with an almost complex circle action with $4$ fixed points. Then one of the following cases occurs:
	
	\begin{enumerate}[label=(\alph*)]
		\item$\mathbb{CP}^3$ type. $\mathrm{Td} (M) = 1$ and the weights of the circle action are $\{a,b,c\}$, $\{-a,b-a,c-a\}$, $\{-b,a-b,c-b\}$, $\{-c,a-c,b-c\}$ for mutually distinct positive integers $a,b,c$.
		
		\item $Q^3$ type.  $\mathrm{Td}(M) = 1$ and the weights of the circle action are $\{a,a+b,a+2b\}$, $\{-a,b,a+ 2b\}$, $\{-a-2b,-b,a\}$, $\{-a-2b,-a-b,-a\}$, for some positive integers $a,b$.
		      
		\item Fano $3$-fold type .  $\mathrm{Td}(M) = 1$ and the weights of the circle action are $\{1,2,3\}$, $\{-1,1, a \}$, $\{-1,1,-a\}$, $\{-1,-2,-3\}$ for some integer $a$.
		      
		\item $S^6 \cup S^6 $ type.  $\mathrm{Td} (M) = 0$ and the weights of the circle action are $\{a,b,-a-b\}$, $\{-a,-b,a+b\}$, $\{c,d,-c-d\}$, $\{-c,-d,c+d\}$, for some positive integers $a,b,c,d$.
		      
		\item $Bl_{p}(S^6)$ type.  $\mathrm{Td}(M) = 0$ and the weights of the circle action are $\{-3a-b,a,b\}$, $\{-2a-b,3a+b,3a+2b\}$, $\{-a,-a-b,2a+b\}$, $\{-b,-3a-2b,a+b\}$, for some positive integers $a,b$ and reversing the circle action if neccesary.
		      
		\item $Bl_{C}(S^6)$ type.  $\mathrm{Td}(M) = 0$ and the weights of the circle action are $\{-a-b,2a+ b, b\}$, $\{-2a-b,a,b\}$, $\{-b,-2a-b,a+b\}$, $\{-a,-b,2a+b\}$, for some positive integers $a,b$.
		      
	\end{enumerate}
\end{theorem}

\section{The Kustarev surgery}\label{Sec: Kustarev surgery}

In this section we will give a presentation of Kustarev's surgery construction from \cite{Ku,Ku2}. We feel that the current
presentation may save some time for readers less familiar with obstruction
theory to verify the details. Moreover, we set up notation to show in some cases the Kustarev construction yields a
unique invariant almost complex structure up to homotopy.

Suppose $M_{0}$ and $M_{1}$ are two
compact, smooth $2n$-manifolds equipped with a smooth faithful action of a connected, compact Lie group 
$G$ of dimension $k$. Assume there are free $G$-orbits $O_{i} \subset M_{i}$.

Then, by the equivariant slice theorem there is a tubular neighborhood $U_{i}$ of $O_{i}$ which is canonically,
equivariant diffeomorphic to $G \times D^{2n-k}$, where $D^{2n-k}$ denotes here
the open disc of radius $1$ in $\RR^{2n-k}$. The action is the standard action on the left factor and the trivial action on the right fact. Thus $U_{i} \setminus O_i
\times \{0\}$ is diffeomorphic to $G \times S^{2n-k-1} \times (0,1)$. We define
the \emph{fibred connected sum} by
\[
	M_{0} \#_{G} M_{1} := (M_{0} \setminus O_0 \times \{0\} \cup
	M_{1} \setminus O_1 \times \{0\})/\sim
\]
where we identify $(t,v,r) \in U_{0} \setminus G \times \{0\} \cong
G \times S^{2n-k-1} \times (0,1)$ with $(t,v, 1-r) \in U_{1} \setminus G 
\times \{0\}$. It is clear that there is an induced $G$-action on the
fiber connected sum.

\begin{theorem}[\cite{Ku,Ku2}]\label{T: Kustarev}
Let $M_{0}$ and $M_{1}$ be two compact manifolds of dimension $2n$, having a smooth faithful almost
complex action of a compact, connected Lie group $G$ of dimension $k$, such that each manifold possess a
free orbit. If $2n-k$ is $2,4,5,6$ mod $8$ then $M_{0}
\#_{G}M_{1}$ admits a $G$-invariant almost complex structure, which
extends the almost complex structures on $M_{i} \setminus U_{i} \subset M_{0}
\#_{G} M_{1}$ for $i=0,1$.
\end{theorem}

It is known that the tangent bundle of a smooth manifold $M$ induces up to
homotopy a unique map $M \to B \mathrm{SO}(2n)$ (we denote by $BH$ the
classifying space of a topological group $H$). The existence of an almost
complex structure on $M$ is equivalent to the existence of a lift $M \to
BU(n)$ of $M \to B \mathrm{SO}(2n)$ through the map $p \colon BU(n) \to
BSO(2n)$. Let $F$ be the homotopy fiber of $p$, thus
$F$ is homotopy equivalent to $\mathrm{SO}(2n)/\mathrm{U}(n)$. We recall a
foundational theorem of obstruction theory, namely the existence and uniqueness
for the relative lifting problem. The existence is proved on \cite[Page 418]{H},
the uniqueness is left as a (straightforward) exercise \cite[Ex. 24, page
420]{H}.

\begin{theorem} \label{toplemma}
	Suppose $p: X \rightarrow Y$ is a map admitting a Moore-Postnikov tower of
	principal fibrations. Suppose that the fibre $F$ of $p$ has abelian
	fundamental group. Suppose there exists a CW pair $(W,A)$, a map $F: W
	\rightarrow Y $ and a $p$-lift $\tilde{F} : A \rightarrow X$.  Then if for all
	$i>0$, $H^{i+1}(W,A,\pi_{i}(F))$ vanishes, then the lift $\tilde{F}$ may be
	extended to a $p$-lift over $W$. Furthermore the extension is unique up to
	homotopy if $H^{i}(W,A; \pi_{i}(F))$ vanishes for all $i>0$.
\end{theorem}
It is known that $p:BU(n) \rightarrow BSO(n)$ admits a Moore-Postnikov tower of principal fibrations, and that the fiber is simply connected. It follows that, Theorem \ref{toplemma}  is applicable to this situation.

The manifold $M_{0} \#_{G} M_{1}$ can be written as the union
\[
	M_{0} \setminus U_{0} \cup V \cup M_{1} \setminus U_{1}
\]
where $V$ is diffeomorphic to $G \times S^{2n-k-1} \times [0,1]$.

Denote by
$J_{i}$ the invariant almost complex structures on $M_{i}$. Then $J_{i}$
restricted to $G \times S^{n-k-1} \times \{i\}$ is determined by $J_{i}$
restricted to $\{\mathbf 1\} \times S^{n-k-1} \times \{i\}$, where $\mathbf 1 
 \in G$ denotes the neutral element. Therefore it
suffices to extend the almost complex structure on $\{\mathbf 1\} \times S^{n-k-1}
\times \{0,1\}$ to an almost complex structure $J'$ on $\{ \mathbf 1\} \times
S^{2n-k-1} \times [0,1]$. This defines an invariant almost complex structure $J$
on $V$ by
\[
	J(t,v,r) := g_{\ast}(J(\mathbf 1, v, r))
\]
where $g_{\ast} \colon TV_{(\mathbf 1, v, r)} \to TV_{(t,v,r)}$ is the
differential of the action of $g \in G$. Note that $J$ coincides with $J_{i}$
restricted to $G \times S^{2n-k-1} \times \{ i\}$ and thus it defines an
invariant almost complex structure on $M_{0} \#_{G} M_{1}$ as claimed in the
theorem. We define now
\[
	W := \{\mathbf 1\} \times S^{2n-k-1} \times [0,1], \quad A = \{\mathbf 1\}
	\times S^{n-k-1} \times \{0,1\}.
\]
Since $(W,A)$ is a good pair (see \cite[p. 124]{H})
 and $W/A$ the
suspension of $S^{2n-k-1}$ we have for $i>0$
\[
	H^{\ast}(W,A;\pi_{i}(F)) \cong H^{\ast}(W/A; \pi_{i}(F)) = H^{\ast}(S^{2n-k};
	\pi_{i}(F)).
\]
For $\pi_{q}(\mathrm{SO}(2n)/U(n))$ does not depend on $n$ provided $q < 2n-1$
  and in this range
$\pi_{q}(\mathrm{SO}(2n)/U(n)) =0$ if $q \equiv 1,3,4,5$ mod $8$. This follows from Bott periodicity, see \cite[Page 432]{G} for the precise statement. Since $k>0$
we have that $\pi_{2n-k-1}(\mathrm{SO}(2n)/U(n))$ does not depend on $n$ and
therefore from Theorem \ref{toplemma} we have the existence of an invariant
almost complex structure if $2n-k \equiv 2,4,5,6$ mod $8$. 

Having recalled Kustarev's construction, we show that under certain conditions on the dimension of the manifold and the group acting the almost complex structure produced is unique.

\begin{theorem}
Suppose that additionally $2n-k = 5,6 \mod 8$. Then, there is a unique homotopy class of almost complex structures extending invariant almost complex structures on $M_{i} \setminus G$ ($i=1,2$).
\end{theorem}
\begin{proof}
In \cite{Ku} it was proven, that the problem of extending the almost complex
structure across the collar $G \times S^{2n-k-1} \times [0,1]$ equivariantly, is
equivalent to the (non-equivariant) relative lifting problem for maps from
$S^{2n-k-1}$ along the map $p$, relative to the boundary. The proof of this,
follows from the fact the group action is free on this region, so an invariant
almost complex structure is determined by its restriction to the slice
$\mathbf{1} \times S^{2n-k-1} \times [0,1]$.

By Theorem \ref{toplemma}, the uniqueness obstruction for the existence property for this lifting property is $$H^{i}(W/A,\pi_{i}(F)) .$$ 

 Since $W/A$ is homeomorphic to $S^{2n-k}$, the existence and uniqueness holds when $$\pi_{2n-k}(F) =  \pi_{2n-k+1}(F) = 0.$$ By the known form of the stable homotopy groups of $SO(2n)/U(n)$  \cite[Page 432]{G},  this happens exactly when $2n-k = 5,6 \mod 8$.

\end{proof}

We note that in the cases which are not covered by the above, the obstruction to
the uniqueness of almost complex structures lies either in $\mathbb{Z}$ or
$\mathbb{Z}_2$, depending on $n$ and $k$. Moreover, the uniqueness proved
above indeed applies to the original construction of Kustarev on $6$-manifolds
with circle action, and more generally.

\begin{corollary}
In case of a circle actions, the invariant almost complex structure of Theorem
\ref{T: Kustarev} on $M_{0} \#_{S^{1}} M_{1}$ is unique if $n$ equals $3$ mod $4$.
\end{corollary}

\section{Proof of Theorem \ref{main}} \label{proofsec}
Before focusing on the Kustarev sum of two copies of $S^6$, and proving Theorem \ref{main}, we state a standard lemma for convenience regarding how fiber connect sum affects the fundamental group and homology of manifolds. A manifold with a circle action is called  $\mathbb{Q}$ resp. $\mathbb{Z}$-equivariantly formal if its equivariant cohomology (with coeffecients in $\mathbb{Q}$ resp. $\mathbb{Z}$) is a free $H^{*}(BS^1)$ module. It is known that if a manifold has a circle action with isolated, and non-empty fixed point set then $M$ being $\mathbb{Q}$ resp. $\mathbb{Z}$ equivariantly formal is equivelant to $H^{i}(M,\mathbb{Q}) = 0$ resp.  $H^{i}(M,\mathbb{Z}) = 0$ for all odd $i$ \cite[Lemma 2.1]{MP}.

\begin{lemma} \label{bet}
Suppose that $M_1,M_2$ are simply connected $6$-manifolds, having almost complex circle actions with isolated, and non-empty fixed point set.  Let $M$ be the Kustarev sum of $M_1$ and $M_{2}$. Then, $M$ is simply connected and, $$H_{3}(M,\mathbb{Z}) = H_{3}(M_{1},\mathbb{Z}) \oplus H_{3}(M_{2},\mathbb{Z}).$$ $$H_{2}(M,\mathbb{Z}) =  H_{2}(M_{1},\mathbb{Z}) \oplus H_{2}(M_{2},\mathbb{Z}) \oplus \mathbb{Z}.$$

In particular, \begin{enumerate}[label=(\alph*)]
\item $M$ is  $\mathbb{Q}$ resp. $\mathbb{Z}$ equivariatly formal $\iff$ both $M_{i}$ are $\mathbb{Q}$ resp. $\mathbb{Z}$ equivariatly formal. 
\item  $\chi(M)=4 $ and $M$ is $\mathbb{Q}$-equivariatly formal $\iff$ both $M_{i}$ are $\mathbb{Q}$-homology $6$-spheres.
\item $\chi(M)=4 $  and $M$ is $\mathbb{Z}$-equivariatly formal $\iff$ both $M_{i}$ are diffeomorphic to $S^6$.
\end{enumerate}
\end{lemma}
\begin{proof}
The proof of the homology formulas, and a)-b) are an elementary applications of Mayer-Vietoris and Seifert-Van-Kampen. To prove c), firstly by Poincar\'{e} duality and the universal coefficient theorem, if $N$ is a simply connected rational homology $S^{6}$, and $0<i<6$ then $H_{i}(N,\mathbb{Z}) \neq 0$ can only happen for $i =2,3$, and furthermore $H_{2}(N,\mathbb{Z}) \cong H_{3}(N,\mathbb{Z}) \cong H^{3}(N,\mathbb{Z})$. Since  each of $M_{i}$ must be a $\mathbb{Q}$-homology $6$-sphere by b) and $\mathbb{Z}$-equivariantly formal by a), both of these two homology groups vanish. That is, each $M_{i}$ is a simply connected integral homology $S^6$, therefore diffeomorphic to $S^6$, for example by the Wall-Jupp-Zubr theorem.
\end{proof}

\subsection{Kustarev sum of two copies of $S^6$} \label{surgery}

In the remainder of this section we prove Theorem \ref{main}, the general overview of the argument is as follows. The Kustarev sum of two manifolds is a fiber connect sum of two manifolds along embedded circles. The fiber connect sum depends on the embedding of the circles up to ambient isotopy and a choice of framing of the embedded circles on both sides.

 In the first subsection, we recall the notion of a framing of an embedded submanifold and the definition of the trivial framing for a sphere embedded in a higher dimensional sphere, following \cite{Ko}. Next, we observe that free orbits for a smooth circle action have a canonical framing which is unique up to homotopy. 

In the following section, we recall the definition of the standard almost complex circle actions on $S^6$, and determine that the framing associated to the free orbits of the actions is the non-trivial one (since $\pi_{1}(GL_{5}(\mathbb{R})) = \mathbb{Z}_2$ it follows that there are two classes of framings, see the first subsection).

Then, we show that the fiber connect sum of two copies of $S^6$ along two subcircles with the non-trivial framing is diffeomorphic to $S^4 \times S^2$. This involves expressing the fiber connect sum as the result of surgery on the $6$-sphere, which is in turn diffeomorphic to the boundary of a certain handle attachment of the $7$-disk which may be determined explicitly.

\subsection{Framings and fiber connect sum}
We follow the definition of fiber connect sum and its notation presented in
\cite[Page 99]{Ko} throughout the remainder of the
article. In the case that $N= S^1$ and $E$ is the trivial vector bundle, the
definition of  \cite[Page 99]{Ko} is as follows:

\begin{definition}\label{fcs}  Let $M_0$ and $M_1$ be two orientable, closed $n$-manifolds and $\varphi_i
		\colon S^1 \times \mathbb{R}^{n-1} \to M_i$ two embeddings.
	Identify $S^1 \times \mathbb{R}^{n-1}\setminus (S^1 \times 0)$ with $S^1
		\times S^{n-2} \times (0,\infty)$ via the standard diffeomorphism. Then we define $M_0 \#_{S^1} M_1$ as 
	\[
		(M_0 \setminus
		\varphi_0(S^1 \times 0) \cup M_1 \setminus \varphi_1(S^1 \times 0))/\sim
	\]
	where $\varphi_0(z,p,t) \sim \varphi_1(z,p,\frac{1}{t})$ for $(z,p,t) \in S^1 \times
		S^{n-2} \times (0,\infty)$.
\end{definition}

The choice of an embedding is known as a framing and is equivalent to a framing of the normal bundle of $\varphi_{i}(S^1 \times \{0\})$. Recall that a framing of the normal bundle is equivalent to a trivialisation of the normal bundle. Because, for $n \geq 3$, $\pi_{1}(GL_{n}(\mathbb{R})) = \mathbb{Z}_2$, when the manifolds have dimension at least $4$ for a given embedded circle there are exactly two choices of framing.

\begin{definition} \label{trivfr}
Consider the embedded $S^1 \subset S^6$  given by $S_{0} = \{(u,v,0,0,0,0,0) \in S^6| u^2 + v^2 =1\}$. Consider the vector subbundle $E$ of $T\mathbb{R}^7$ whose fiber over each point is the linear subspace $\{x\in \mathbb{R}^7 |x_{1}=x_{2}=0\}$. Then, note that the normal bundle $N$ of $S_0$ is the restriction of $E$ to $S_0$. The \textit{trivial framing} of $S_0$ is defined to be the framing of $N$ given by $e_{3}, \ldots, e_{7}$ over every point, where $e_{i}$ denotes the standard basis of $\mathbb{R}^7$.
\end{definition}

Because, any two embeddings of $S^1$ to $S^6$ are ambient isotopic, and ambient isotopies respect homotopy types of framings, a framing of an embedded $S^1$ in $S^6$ is called trivial when it is ambient isotopic to the one from Definition \ref{trivfr}.

\subsection{Uniqueness of equivariant framings}

In this section we consider a natural framing of free orbits a manifold with a circle action,  and show that it is unique up to homotopy. In the coming section, we will determine this homotopy type of the framing for almost complex circle action on $S^6$.

\begin{definition}
Let $M$ be a manifold with a circle action, and $O$ a free orbit. Then, a framing $F$ of $O$ is called \textbf{equivariant} if it is of the from $$g_{*}(v_{1}), \ldots g_{*}(v_{k}),$$ where $v_{1},\ldots,v_{k}$ is a basis of the normal bundle of $O$ at a point in $O$.   
\end{definition}

In the following lemma, we show that equivariant framings yield a unique homotopy class:

\begin{lemma}
Suppose that $M$ is an $n$-manifold with a smooth $S^1$-action.  Then, for each free orbit $O$ any two equivariant framings of $O$ are homotopic.
\end{lemma}
\begin{proof}
Suppose for a contradiction that there were two such non-homotopic framings. Then, by the equivariant slice theorem there would exist an $S^1$-equivariant self-diffeomorphism $\phi$ of $D^{n-1} \times S^1$, with action $z.(p,w) = (p,zw)$ which maps the class of trivial framing of $\{0\} \times S^1$ to the non-trivial one. But, choosing $F$ to be a constant framing, and applying the $S^1$-equivariance of $\phi$ directly shows that $\phi_{*}F$ is a constant framing, hence represents the same homotopy class as $F$, which is the desired contradiction.

\end{proof}

\subsection{Computation of equivariant framing}\label{comfram}
In this subsection, we compute the class of the equivariant framing for the standard almost complex $S^1$-actions on $S^6$. We will use this is in the next subsection to determine the diffeomorphism type of the Kustarev sum of two copies of $S^6$.

\begin{theorem} \label{nonframe}
Consider the standard almost complex $S^1$ actions on $S^6$. Then, the homotopy class of the equivariant framing of a free orbit is the non-trivial one.
\end{theorem}
\begin{proof}
We recall the definition of  almost complex $S^1$-actions on $S^6$ following \cite[Example 1.9.1]{GZ}. It holds that $S^6 = G_2/ \mathrm{SU}(3)$ and $S^6$ admits an invariant almost complex structure. Let $T^2 \subset G_{2}$ be a Cartan subgroup, then the induced action on $S^6$ may be realized as the restriction of a linear action on $\mathbb{R}^7$ as follows: For $(t,s) \in T^2 = S^1 \times S^1$ and $(z_1,z_2,z_3,x) \in S^6$ we have
	\[
		(t,s) \cdot (z_1,z_2,z_3, x) = (t^{-1}z_1,sz_2,(ts)z_3,x).
	\]
	This action has two isolated fixed points $(0,0,0,\pm 1)$. Consider the
	subcircle $S^1 \to T^2$, $t \mapsto (t^a,t^b)$, $a,b \in \mathbb{Z}$, $a,b \neq 0$, which gives an almost complex circle action on $S^6$:
	\[
		t \cdot (z_1,z_2,z_3, x) = (t^{-a}z_1,t^bz_2,t^{a+b}z_3,x).
	\]

To determine the class of the equivariant (normal) framing we note that for a
given embedded circle $C$ in $S^6$, there is a natural map $\Psi$ sending normal
framings $F$ of $C$ in $S^6$ to framings of $T \mathbb{R}^7|_C$, given by $$\Psi( F(t)) = [c(t),\dot{c}(t),F(t)] .$$

It is not hard to check that this map is invariant with respect to ambient isotopy, by extending the ambient isotopy of $S^6$ radially on a neighborhood of $S^6 \subset \mathbb{R}^7$. 

Note that a tangent framing of a circle in $\mathbb{R}^7$ is the same as a map $S^1 \rightarrow GL_{7}(\mathbb{R})$, after choosing a point on the circle and a basis of the tangent space at that point. Moreover the class in $\pi_{1}(GL_{7}(\mathbb{R}))$ is independent of such choices. Such a tangent framing is called trivial if it represents the trivial element in  $\pi_{1}(GL_{7}(\mathbb{R}))$.

Determining the class of the equivariant normal framing will be done by computing the image of this class by $\Psi$. First, it is necessary to determine how $\Psi$ acts on homotopy classes of framings. Because, $\Psi$ is invariant with respect to ambient isotopy, it is necessary just to do this for a given choice of embedded circle and framing. Taking $C$ to be a standard coordinate subcircle in the first two co-ordinates with trivial framing $e_{3},\ldots,e_7$.

Then, the image under $\Psi$ under this framing, is the framing associated to the map $\phi: S^1 \rightarrow GL_{7}(\mathbb{R})$, $$\phi(t) =  \begin{pmatrix}
			R(t) & 0       \\
			0    & I_{5} \\
		\end{pmatrix}$$ 

Where $R(t)$ is the standard 2 by 2 rotation matrix associated to $t \in S^1$, and $I_{5}$ is the 5 by 5 identity matrix. Clearly, this represents the non-trivial element in  $\pi_{1}(GL_{7}(\mathbb{R}))$. Therefore $\Psi$ maps the class of the trivial normal framing in $S^6$ to the class of the non-trivial tangent framing in $\mathbb{R}^7$, an almost identical argument shows it sends the non-trivial class to the trivial one.

Finally, it remains to determine the class of the image of the equivariant normal framing of this circle action under $\Psi$. To do this, note that because the circle action extends to a linear circle action on $\mathbb{R}^7$, the way that this action acts on tangent vectors is exactly by the linear action itself. 

Picking a point $p=c(0)$ on $C$ and defining a continuous path of transformations of $T\mathbb{R}^7_{p}$ which deform a basis of the form $c(0), \dot{c}(0), F(0)$ to the standard basis of $\mathbb{R}^7$, shows that this class of the image of the equivariant normal framing under $\Psi$, is the class of the map $F: S^1 \rightarrow GL_{7}(\mathbb{R})$,

$$F(t) = \begin{pmatrix}
			R_{-a}(t) & 0      & 0   & 0  \\
			0      & R_b(t) & 0   & 0  \\
			0      & 0      & R_{a+b}(t) & 0  \\
			0      & 0      & 0   & 1
		\end{pmatrix} $$

Finally, because $-a + b + a + b = 2b$ is even, the class of this loop is trivial in $\pi_{1}(GL_{7}(\mathbb{R})).$ That is the image under $\Psi$ of the equivariant normal framing is the trivial one, showing by the above considerations about $\Psi$, that the class of the equivariant normal framing is the non-trivial one.

\end{proof}

\subsection{Determining the diffeomorphism type: Surgery and Handle attachments}

The following Lemma shows that the fiber connect sum of two copies of $S^6$ with the non-trivial framing and the fiber connect sum of two copies of $S^6$ with the trivial framing are diffeomorphic.

\begin{lemma} \label{revfr}
Consider an embedded circle $C$ in $S^6$. Let $h_{1}$ denote the trivial framing of $C$ and $h_{2}$ the non-trivial framing. Then, the fiber connect sum $(S^6,C,h_{1}) \# (S^6,C,h_{1})$ is diffeomorphic to the fiber connect sum  $(S^6,C,h_{2}) \# (S^6,C,h_{2})$.
\end{lemma}
\begin{proof}
Fix an embedding $\phi : S^1 \times \mathbb{R}^5 \rightarrow N_{c}$, corresponding to the trivial framing of $C$, where $N_C$ is a neighbourhood of $C$ in $S^6$. 

Set coordinates of $S^1 \times \mathbb{R}^5$ by identifying is as a subset of $\mathbb{C}^3 \times \mathbb{R}$, given by $\{(z_{1},z_2,z_3,t) \in \mathbb{C}^3 \times \mathbb{R} | |z_{1}|=1\}.$ Define a map $h: S^1 \times \mathbb{R}^5 \rightarrow S^1 \times  \mathbb{R}^5$ given by the formula $$h(z_1,z_{2},z_{3},t ) = (z_1,z_1 z_{2},z_{3},t ) . $$

Note that $h$ is a diffeomorphism with inverse $h^{-1}(z_1,z_{2},z_{3},t ) = (z_1,z_1^{-1} z_{2},z_{3},t ) . $

Then, $\phi \circ h : S^1 \times \mathbb{R}^5 \rightarrow N_{c}$ is an embedding corresponding to the non-trivial framing of $C$. Following the definition of fiber connect sum, set $\alpha: (0,\infty) \rightarrow (0,\infty)$ an orientation reversing diffeomorphism, and set $\alpha_{E} :  (S^1 \times \mathbb{R}^5) \setminus (S^1 \times \{0\}) \rightarrow  (S^1 \times \mathbb{R}^5) \setminus (S^1 \times \{0\})$; $$\alpha_{E}(z_{1},z_{2},z_{3},t) = (z_{1},0,0,0) +  \frac{\alpha(|(z_2,z_3,t)|)}{|(z_2,z_3,t)|} (0,z_2,z_3,t).$$

By definition $(S^6,C,h_{1}) \# (S^6,C,h_{1})$ is diffeomorphic to two copies of $S^6 \setminus C$ glued along $N_{C} \setminus C$ by $\phi \circ \alpha_{E} \circ \phi^{-1}$. Similarly $(S^6,C,h_{2}) \# (S^6,C,h_{2})$ is diffeomorphic to two copies of $S^6 \setminus C$ glued along $N_{C} \setminus C$ by $\phi \circ h \circ \alpha_{E} \circ h^{-1} \circ \phi^{-1}$. Then, one may check directly from the formulas that  $$h \circ \alpha_{E} \circ h^{-1} = \alpha_E.$$ The gluing maps are equal and the two spaces are diffeomorphic.

\end{proof}

Since, we have proven in the previous section, that the Kustarev sum $M$ is the fiber connect sum of two copies of $S^6$ with the non-trivial framing. By Lemma \ref{revfr} it is diffeomorphic to the fiber connect sum of two copies of $S^6$ with the trivial framing. 

Next, the following fact asserted in \cite[Page 112]{Ko} shows that $M$ is diffeomorphic to the surgery on $S^6$ along an embedded circle with the trivial framing.

\begin{lemma} \label{surgery}  \cite[Page 112]{Ko} Let M be an $n$- manifold with an embbedded sphere $S$ and framing $h$. The fiber connect sum of $M$ with $S^n$, where $S \subset M$ has framing $h$ and the co-ordinate subsphere of $S^n$ has the trivial framing, is diffeomorphic to the surgery on M along $(S,h)$.
\end{lemma}

Therefore $M$ is diffeomorphic to the result of the surgery on $S^6$ along $M$ with the trivial framing.

\begin{lemma} \label{final}

The surgery on $S^6$ along $S^1$ with the trivial framing, is diffeomorphic to $S^4 \times S^2$.

\end{lemma}
\begin{proof}
Consider the disk $D^7$, and attach a $2$-handle to $\partial D^7 = S^6$ via the trivial framing. Then, by \cite[Example 4.1.4 (d)]{GS}, the resulting manifold with boundary is $D^5 \times S^2$. It is known that the boundary of the handle attachment along an embedded sphere in the boundary is diffeomorphic to the surgery on the same sphere in the boundary, with the same framing (\cite[Page 154]{GS}, \cite[page 112]{Ko}). Therefore, the surgery on $S^6$ along $S^1$ with trivial framing is diffeomorphic to the boundary of $D^5 \times S^2$, namely $S^4 \times S^2$. 
\end{proof}

\begin{proof}[Proof of Theorem \ref{main}]
The Kustarev sum of two copies of $S^6$, is by definition the fiber connect sum of two copies of $S^6$ along circles with the equivariant framing on both sides, denote this manifold by $M$. By Theorem \ref{nonframe} the equivariant framing is the non-trivial one. By Lemma \ref{revfr} $M$ is diffeomorphic to the the fiber connect sum of two copies of $S^6$ along circles with the trivial framing on both sides. By Lemma \ref{surgery} $M$ is diffeomorphic to the result of sugery of $S^6$ along an embedded circle with the trivial framing. By Lemma \ref{final}, $M$ is diffeomorphic to $S^4 \times S^2$.
\end{proof}

\subsection{Proof of Corollary \ref{exotic}}

In the following, we consider the Kustarev sum of two copies of $S^6$, with almost complex circle actions with weights $$ \{a,b,-a-b\},\{-a,-b,a+b\}$$ and $$ \{c,d,-c-d\},\{-c,-d,c+d\},$$ respectively (see \cite[Example 1.9.1]{GZ} for the construction of circle actions on $S^6$.). In the following we choose the parameters so that we have $b>a>1$, and $d>c>1$.

\begin{proof} [Proof of Corollary \ref{exotic}]

Suppose for a contradiction that the Kustarev sum is equivariantly diffeomorphic to a linear action. Let $\Phi$ denote the equivariant diffeomorphism.

For integers $w_1,w_2,w_3$ linear $S^1$ actions on $S^4 \times S^2$,  are given by $S^4 \times S^2 \subset (\mathbb{C}^2 \oplus \mathbb{R}) \times (\mathbb{C} \oplus \mathbb{R})$, given by $$z.((u_1,u_t,t), (v,s) ) =  ((z^{w_1} u_1,  z^{w_{2}} u_2,t), (z^{w_{3}} v,s) ).$$ 

The fixed point set of the action consists of the four points of the form $((0,0,\pm 1) , (0,\pm1))$. Since the action on the Kustarev sum has isolated fixed points the linear action has isolated fixed points. Therefore, without loss of generality $w_{i}>0$ for each $i$.

 Let $p$ be a fixed point which is mapped to $((0,0,-1),(0,-1))$. The $\mathbb{R}$ weights of the linear action at this fixed point $w_{1},w_{2},w_{3}$ and the $\mathbb{R}$ weights of the Kustarev sum are $\{a,b,a+b\}$, where $a,b>1$. Because, $\Phi$ is an equivariant diffeomorphism on neighbourhoods of these fixed points, the weights are equal. To see this one can pick an invariant Riemannian metric and consider the exponential at each of the fixed points respectively. In particular, $w_{i}>1$ for each $i$, and $w_{i}$ are pairwise coprime.

 Next, since each $w_{i}$ is greater than one, there are $6$ isotropy spheres of the linear action, of weight $w_{1},w_{2}$ containing $\{((0,0,-1),(0,1)),((0,0,1),(0,1))\}$ and containing $\{ ((0,0,-1),(0,-1)) , ((0,0,1),(0,-1)) \}$.

 Also,  there exist isotropy spheres of weight $w_{3}$ containing $\{ ((0,0,1),(0,-1)), ((0,0,1),(0,1)) \}$ and containing
 $\{ ((0,0,-1),(0,-1)),
 ((0,0,-1),(0,1)) \}$.

In particular, for the linear action the subset of points with non-trivial stabilizer group is connected. However, since the Kustarev sum is performed on a neighborhood of a free orbit, it is clear that for the action on the Kustarev sum the set of points with non-trivial stabilizer group is disconnected, a contradiction.
\end{proof}

\section{Some Topological consequences of Jang's Classification}

In this section, we prove a topological consequence of Jang's classification where the manifold is simply connected and $\mathbb{Z}$-equivariantly formal and fits into Case 6 of Jang's classification (also called the $Bl_{C}(S^6)$ case). 

\begin{theorem} \label{unkquad}
	Suppose that $M$ is a simply connected, almost complex $6$-manifold and an almost complex circle action fitting into case $6$ of Jang's classification. Moreover, suppose that it is $\mathbb{Z}$-equivariantly formal. Then, $M$ is diffeomorphic to a Quadric $Q \subset \mathbb{CP}^4$.
\end{theorem}
\begin{proof}
	
	Since the circle action has $4$ fixed points, $\chi(M) = 4$. Then, combining this condition with $\pi_{1}(M)=1$,  and $H^{3}(M,\mathbb{Z} )= 0 $, $  H_{2}(M,\mathbb{Z})= \mathbb{Z}$, a routine application of the universal coefficients theorem shows that $H^{i}(M,\mathbb{Z}) = \mathbb{Z}$ for $i=0,2,4,6$ and  $H^{i}(M,\mathbb{Z}) = 0$  otherwise.
	
	By Lemma \ref{Degreeformula}  $c_1(J)^3 = -2$. In particular this implies that $w_{2}(M) $ is the non-trivial element of $H^{2}(M,\mathbb{Z}_2) = \mathbb{Z}_2$, because otherwise $c_{1}$ is even so $c_{1}^3$ would be divisible by $8$. Also by Lemma \ref{prel} $c_1 c_2=0$.  Let $\alpha=-c_1 \in H^{2}(M,\mathbb{Z})$, note that $\alpha$ is a generator of  $H^{2}(M,\mathbb{Z})=\mathbb{Z}$. Then by Poincare duality $c_1c_2=0$ implies that $c_2=0$. Finally, using the equation $p_1 = c_1^2-2c_2$, gives that $p_1 = \alpha^2$. Hence we have shown that the cohomology ring and Pontryagin, and second Stiefel-Whitney classes are isomorphic to the Quadric, the result follows from the Wall-Jupp-Zubr theorem.
\end{proof}

\section{Some questions and further directions}
In \cite{Ku}, Kustarev originally utilized his surgery operation to show that there are infinitely many diffeomorphism types of almost complex $6$-manifolds having a almost complex circle action with two fixed points. He showed this by taking the Kustarev sum of $S^6$, and an almost complex $6$-manifold of the form $M^4 \times T^2$, with a free almost complex circle action on the $T^2$-factor. This construction shows that the fundamental group varies through infinitlely many isomorphism types.

It appears however, that the following question about topological types of almost complex $6$-manifolds with two fixed points remained open.

\begin{question} \label{homologysphere}
Suppose that $M$ is a closed, almost complex, simply connected $\mathbb{Q}$-homology $S^6$ with an almost complex circle action with two fixed points. Is $M$ diffeomorphic to $S^6$?
\end{question}

We note that without the assumption of a circle action there is a wealth of almost complex, simply connected, rational homology $6$-spheres.  

We also pose a question about actions with $4$ fixed points. Similarly to the above, in each of cases 1-4 of Jang's classification \cite[Theorem 1.1]{Ja} taking a Kustarev sum with almost complex $6$-manifolds with manifolds having free circle actions, shows that there are infinitely diffeomorphism types and furthermore $\pi_1$ varies through infinitlely many isomorphism types.

In the following question we pose a finiteness question  for simply connected, equivariantly formal manifolds appearing in Jang's classification.

\begin{question} \label{finite} \begin{enumerate}
\item[a)] Are there finitely many diffeomorphism types of, simply connected, $\mathbb{Z}$-equivariantly formal, almost complex 6-manifolds having a circle action with 4 fixed points?

\item[b)] Are there finitely many homotopy types of simply connected, $\mathbb{Z}$-equivariantly formal, almost complex 6-manifolds having a circle action with 4 fixed points?
\end{enumerate}
\end{question}

The existence of the new example $S^4 \times S^2$ of Theorem \ref{main}, which fits into the above class but is not covered by previous classifications \cite{A,To2}, perhaps gives additional motivation to consider Question \ref{finite}.
  

Contact: pako@mathematik.uni-marburg.de, nlindsay@math.uni-koeln.de.


\begin{thebibliography}{19}

\bibitem{A} K. Ahara: 6-dimensional almost complex $S^1$-manifolds with $\chi(M) = 4$. J. Fac. Sci.
Univ. Tokyo Sect. IA, Math. 38  no 1, 47-72 (1991).

\bibitem{AB} M. F. Atiyah and R. Bott. The moment map and equivariant cohomology. Topology, Vol. 23, No.1 pp 1-28 (1984). 


\bibitem{BV} N. Berline. M. Vergne. Classes caracteristiques equivariantes. Formule de localisation en cohomologie equivariante, C.R.Acad.Sci. Paris 295 (1982) 539--541.

\bibitem{G} J. W.Gray. 
Some global properties of contact structures .Ann. of Math. (2)69 421–450 (1959).

\bibitem{GS} R.E. Gompf, A.I. Stipsicz. $4$-manifolds and Kirby Calculus. Graduate Studies in Mathematics, Volume 20. American Mathematical Society (1999).

\bibitem{GZ} V. Guillemin, C. Zara. 1-skeleta, Betti numbers, and equivariant cohomology.  Duke Math. J. 107(2): 283-349 (2001). 


\bibitem{H} A. Hatcher. Algebraic Topology. Cambridge University Press (2002).

\bibitem{HBJ} F. Hirzebruch. T. Berger. R. Jung. Manifolds and Modular Forms. Aspects of Mathematics (1992).


\bibitem{Ja} D. Jang. Circle actions on almost complex manifolds with 4 fixed points. Mathematische Zeitschrift volume 294, pages287–319 (2020).

\bibitem{Ko} A.A. Kosinski. Differentiable Manifolds. Academic Press Inc. (1992).

\bibitem{Ku} A.  Kustarev. Almost complex circle actions with few fixed points, Russian Mathematical Surveys, Volume 68, Issue 3, 574–576, (2013).

\bibitem{Ku2} A.  Kustarev. Chern numbers of manifolds with torus action. arXiv:1506.05355.


\bibitem{M}W. S. Massey. 
Obstructions to the existence of almost complex structures. Bull. Amer. Math. Soc.67, 559–564 (1961).

\bibitem{MP}M.  Masuda, T. Panov On the cohomology of torus manifolds. Osaka J. Math.43 , 711–746 (2006).



\bibitem{To2} S. Tolman. A Symplectic Generalization of Petries Conjecture. Transactions of the American Mathematical Society 362 (08): 3963--3996, (2010).




\end{thebibliography}

%
%
%
%
%
%
%
%
%
%
%
%

\end{document}